\documentclass[reqno]{amsart}

\usepackage{latexsym}
\usepackage{amsmath}
\usepackage{amssymb, amscd, amsthm}
\usepackage[all]{xy}
\usepackage{multirow}

\newtheorem{thm}{{Theorem}}

\newtheorem{lemma}[thm]{{Lemma}}
\newtheorem{prop}[thm]{Proposition}
\theoremstyle{remark}

\newtheorem*{remark}{Remark}

\newcommand{\Z}{\mathbb{Z}}
\newcommand{\F}{\mathbb{F}}
\newcommand{\C}{\mathbb{C}}
\newcommand{\Q}{\mathbb{Q}}

\newcommand{\ra}{\rightarrow}

\newcommand{\g}{\gamma}
\newcommand{\diag}{\mathrm{diag}}
\newcommand{\Tr}{\mathrm{Tr}}
\newcommand{\NP}{\mathrm{NP}}
\newcommand{\HP}{\mathrm{HP}}
\newcommand{\gap}{\mathrm{gap}}
\newcommand{\detphi}{\det\nolimits^{\varphi^{-1}}}
\newcommand{\udots}{\mathinner{\mskip1mu\raise1pt\vbox{\kern7pt\hbox{.}}
\mskip2mu\raise4pt\hbox{.}\mskip2mu\raise7pt\hbox{.}\mskip1mu}}

\makeatletter
\@namedef{subjclassname@2010}{%
  \textup{2010} Mathematics Subject Classification}
\makeatother

\begin{document}
\title[Newton polygons]{Newton polygons of $L$-functions of polynomials $x^d+ax^{d-1}$ with $p\equiv-1\bmod d$}
\author{Yi Ouyang, Shenxing Zhang}
\address{Wu Wen-Tsun Key Laboratory of Mathematics, School of Mathematical Sciences, University of Science and Technology of China, Hefei, Anhui 230026, PR China}
\email{yiouyang@ustc.edu.cn, zsxqq@mail.ustc.edu.cn}
\subjclass[2010]{Primary 11; Secondary 14}
\thanks{Corresponding author: S. Zhang. Email: zsxqq@mail.ustc.edu.cn}

\begin{abstract}
For prime $p\equiv-1\bmod d$ and $q$ a power of $p$, we obtain the slopes of the $q$-adic Newton polygons of $L$-functions of $x^d+ax^{d-1}\in \F_q[x]$ with respect to finite characters $\chi$ when $p$ is larger than an explicit bound depending only on $d$ and $\log_p q$. The main tools are Dwork's trace formula and Zhu's rigid transform theorem.
\end{abstract}

\maketitle

\section{Main results}
Let $q=p^h$ be a power of the rational prime number $p$. Let $v$ be the normalized valuation on $\overline{\Q}_p$ with $v(p)=1$. For a polynomial $f(x)\in\F_q[x]$, let $\hat f\in\Z_q[x]$ be its Teichm\"uller lifting. For a finite character $\chi:\Z_p\ra\C_p^\times$ of order $p^{m_\chi}$, define the
$L$-function
  \begin{equation} L^*(f,\chi,t)=\exp\left(\sum_{m=1}^\infty S_m^*(f,\chi)\frac{t^m}{m}\right),   \end{equation}
where $S^*_m(f,\chi)$ is the exponential sum
  \begin{equation} S^*_m(f,\chi)=\sum_{x\in\mu_{q^m-1}}\chi(\Tr_{\Q_{q^m}/\Q_p}\hat f(x)) \end{equation}
and $\mu_n$ is the group of $n$-th roots of unity. Then $L^*(f,\chi,t)$ is a polynomial of degree $p^{m_\chi-1}d$ by Adolphson-Sperber \cite{AS} and Liu-Wei \cite{LWe}. We denote $\NP_q(f,\chi,t)$ the $q$-adic Newton polygon of $L^*(f,\chi,t)$.

We fix a  character $\Psi_1:\Z_p\ra\C_p^\times$ of order $p$, and denote $L^*(f,t)=L^*(f,\Psi_1,t)$ and $\NP_q(f,t)=\NP_q(f,\Psi_1, t)$. When $p\equiv  1\bmod d$,  it is well-known that $\NP_q(f,t)$ coincides the Hodge polygon with slopes $\{i/d:0\le i\le d-1\}$.

Let $a$ be a nonzero element in $\F_q$. For $f(x)=x^d+ax^s (s<d)$, Liu-Niu and Zhu obtained the slopes of $\NP_q(f,t)$ for $p$ large enough under certain conditions in \cite[Theorem~1.10]{LN2} and \cite{Z2}, but these conditions are not so easy to check. For $f(x)=x^d+ax$, Zhu, Liu-Niu and Ouyang-J.\ Yang obtained the slopes in \cite[Theorem~1.1]{Z2}, \cite[Theorem~1.10]{LN} and \cite[Theorem~1.1]{OY}, see also R.\ Yang \cite[\S 1 Theorem]{Y} for earlier results.

In \cite{DWX}, Davis--Wan--Xiao gave a result on the behavior of the slopes of $\NP_q(f,\chi,t)$ when the order of $\chi$ is large enough. In this way for $p$ sufficiently large, they can obtain the slopes of $\NP_q(f,\chi, t)$ based on the slopes of $\NP_q(f,\chi_0, t)$ with $\chi_0$ a character of order $p^2$. In \cite{N}, Niu gave a lower bound of the Newton polygon $\NP_q(f,\chi, t)$. In \cite[Theorem~4.3]{OY}, Ouyang--Yang showed that if the Newton polygon of $L^*(f,t)$ is sufficiently close to its Hodge polygon, the slopes of $\NP_q(f,\chi,t)$ for $\chi$ in general follow from the slopes of $\NP_q(f,t)$. As a consequence they obtained the slopes of $\NP_q(x^d+ax,\chi,t)$ when $p$ is bigger than an explicit bound depending only on $d$ and $h$.

Our main results are the following two theorems.
\begin{thm}\label{slope}
Let $f(x)=x^d+ax^{d-1}$ be a polynomial in $\F_q[x]$ with $a\neq 0$. Let $N(d)=\frac{d^2+3}{4}$ for $q=p$ and $\frac{d^2}{2}$ for general $q$. If $p\equiv-1\bmod d$ and $p> N(d)$, the $q$-adic Newton polygon of $L^*(f,t)$ has slopes
  \[\{w_0,w_1,\ldots,w_{d-1}\},\]
where
  \[ w_i=\begin{cases}
      \frac{(p+1)i}{d(p-1)},&\ \text{if}\ i<\frac{d}{2};\\
      \frac{(p+1)i-d}{d(p-1)}=\frac{1}{2},&\ \text{if}\ i=\frac{d}{2};\\
      \frac{(p+1)i-2d}{d(p-1)},&\ \text{if}\ i>\frac{d}{2}.
  \end{cases} \]
\end{thm}
\begin{remark}
(1) For general $p$, write $pi=dk_i+r_i$ with $1\le i,r_i\le d-1$. If $r_i>s$ for any $1\le i\le s$, then one can decide that the first $s+1$ slopes of $\NP_q(f,t)$ are  $\{0,\frac{k_1+1}{p-1},\ldots,\frac{k_s+1}{p-1}\}$ by our method for sufficiently large $p$. For the rest of slopes, one needs to calculate the determinants of submatrices of  ``Vandermonde style'' matrices.

(2) The slopes in our case coincide Zhu's result in \cite{Z2}.
\end{remark}

\begin{thm}\label{chislope}
Assume $f(x)$ and $N(d)$ as above. For any non-trivial finite character $\chi$, if $p\equiv-1\bmod d$ and $p>\max\{N(d),\frac{h(d^2-1)}{4d}+1\}$, the $q$-adic Newton polygon of $L^*(f,\chi,t)$ has slopes
  \[\{p^{1-m_\chi}(i+w_j):0\le i\le p^{m_\chi-1}-1,0\le j\le d-1\}.\]
\end{thm}

\section{Preliminaries}
\subsection{Dwork's trace formula}
We will recall Dwork's work for $f(x) =x^d+ax^{d-1}$. For general $f$, one can see \cite[\S 2]{OY}.

Let $\g\in \Q_p(\mu_p)$ be a root of the Artin-Hasse exponential series
  \[E(t)=\exp(\sum_{m=0}^\infty p^{-m}t^{p^m})\]
such that $v(\g)=\frac{1}{p-1}$. Fix a $\g^{1/d}\in\bar\Q_p$. Let
  \[ \theta(t)=E(\g t)=\sum_{m=0}^\infty \g_m t^m \]
be Dwork's splitting function. Then $v(\g_m)\ge m/(p-1)$, and $\g_m=\g^m/m!$ for $0\le m\le p-1$. Let
  \[ F(x)=\theta(x^d)\theta(ax^{d-1})=\sum_{i=0}^\infty F_i x^i, \]
then
  \[ F_i=\sum_{dm+(d-1)n=i} \g_m\g_n a^n.\]
One can see $m+n\ge i/d$ and $v(F_i)\ge \frac{i}{d(p-1)}$.

Set $A_1=(F_{pi-j}\g^{(j-i)/d})_{i,j\ge 0}$. This is a nuclear matrix over $\Q_q(\g^{1/d})$ with
  \[ v(F_{pi-j}\g^{(j-i)/d})\ge \frac{pi-j}{d(p-1)}+\frac{j-i}{d(p-1)}=\frac{i}{d}. \]
We extend the Frobenius $\varphi$ to $\Q_q(\g^{1/d})$ with $\varphi(\g^{1/d})=\g^{1/d}$.

\begin{thm}[Dwork]\label{Dwork}
Let $A_h=A_1\varphi(A_1)\cdots\varphi^{h-1}(A_1)$. Then
  \[ L^*(f,t)=\frac{\det^{\varphi^{-1}}(I-tA_h)}{\detphi(I-tqA_h)}. \]
\end{thm}

\subsection{Zhu's rigid transformation theorem}
Let $U_1=(u_{ij})_{i,j\ge0}$ be a nuclear matrix over $\Q_q(\g^{1/d})$. Then the Fredholm determinant $\det(I-tU_1)$ is well defined and $p$-adic entire (see \cite{S}). Write
  \[ \det(I-tU_1)=c_0+c_1t+c_2t^2+\cdots.\]
For $0\le t_1<t_2<\cdots<t_s$, denote by $U_1(t_1,\ldots,t_s)$ the principal sub-matrix consisting of $(t_i,t_j)$-entries of $U_1$ for $1\le i,j\le s$. In particular, denote $U_1[s]=U_1(0,1,\ldots,s-1)$. Then we have $c_0=1$ and for $s\ge 1$,
  \[ c_s=(-1)^s\sum_{0\le t_1<t_2<\cdots<t_s} \det U_1(t_1,t_2,\ldots,t_s).\]
Let $U_h=U_1\varphi(U_1)\cdots\varphi^{h-1}(U_1)$. Write
  \[ \det(I-tU_h)=C_0+C_1t+C_2t^2+\cdots.\]

\begin{thm}(See \cite[Theorem~5.3]{Z1}.)\label{zhu}
Suppose $(\beta_s)_{s\ge 0}$ is a strictly increasing sequence such that
  \[ \beta_i\le v(a_{ij})\ \text{and}\ \lim_{s\ra +\infty} \beta_s=+\infty. \]
If
  \[ \sum_{s<i} \beta_s\le v(\det U_1[i])\le \frac{\beta_i-\beta_{i-1}}{2}+\sum_{s<i}\beta_s \]
holds for every $1\le i\le k$, then $v(C_i)=hv(\det U_1[i])$ for $1\le i\le k$ and
  \[\NP_q(\det(I-tA_h[k]))=\NP_p(\det(I-tA_1[k])).\]
\end{thm}

\section{Slopes of the Newton polygon of $L^*(f,\chi,t)$}
From now on, we assume $p\equiv-1\bmod d$ and write $p=dk-1$.
\subsection{The case $\chi=\Psi_1$}
\begin{lemma}
Let $M(s)=(a_{ij})_{1\le i,j\le s}$ be an $s\times s$ matrix with entries
  \[ a_{i,j}=\frac{a^{i+j}}{(ki-i-j)!(i+j)!}.\]
Then $v(\det M(s))=0$ for $1\le s\le d-1$.
\end{lemma}
\begin{proof}
Denote $x[0]=1$ and $x[n]:=x(x-1)\cdots(x-n+1)$ for $n\geq 1$. Then $x[n]$ is a polynomial of $x$ of degree $n$ and $\{(x+j)[t]:0\le t\le j-1\}$ is a basis of the space of polynomials of degree $\le j-1$.  Thus we can write
  \[((k-1)x-1)[j-1]=c_0(j)+\sum_{t=1}^{j-1}c_t(j)\cdot (x+j)[t].\]
Let $x=-j$, we get
  \[c_0(j)=((k-1)(-j)-1)[j-1]=((1-k)j-1)[j-1].\]
For any $1\le u\le j-1$,
  \[ 1\le(k-1)j+u<kj\le k(d-1)\le p. \]
Hence $p\nmid (1-k)j-u$ and $v(c_0(j))=0$.

Let $D=\diag\{a,a^2,\ldots,a^s\}$ and $M'=(a'_{ij})_{1\leq i,j\leq s}$ with $a'_{ij}=a_{ij}a^{-i-j}$, then
  \begin{equation}\label{mateq0} M(s)=D M' D.  \end{equation}

Let $a''_{ij}:=(ki-i-1)!(i+s)!a'_{ij}$. Then
  \[\begin{split}
     a''_{ij}&=(ki-i-1)[j-1]\cdot(i+s)[s-j]\\
            &=\sum_{t=0}^{j-1}c_t(j)\cdot (i+j)[t]\cdot (i+s)[s-j],\\
            &=\sum_{t=0}^{j-1}c_t(j)\cdot (i+s)[s-j+t]\\
            &=\sum_{t=1}^{j} (i+s)[s-t] \cdot c_{j-t}(j).
  \end{split}\]
Define $c_{j-t}(j):=0$ for $j<t$. Write $M''=(a''_{ij})_{1\le i,j\le s}$,  $M_1=((i+s)[s-t])_{1\leq i,t\leq s}$ and $M_2=(c_{j-t}(j))_{1\leq t,j\leq s}$. Then
  \begin{equation}\label{mateq1} M''=M_1 M_2.  \end{equation}

Write
  \[ x[n]=\sum_{t=0}^{n}c'_t(n)x^t, \]
then $c'_n(n)=1$ and
  \[ (i+s)[s-j]=\sum_{t=0}^{s-j}c'_t(s-j)(i+s)^t. \]
Define $c'_t(n):=0$ for $t>n$. Write $M_{11}=((i+s)^{t-1})_{1\leq i,t\leq s}$ and $M_{12}= (c'_{t-1}(s-j))_{1\leq t,j\leq s}$. Then
  \begin{equation}\label{mateq2} M_1=M_{11} M_{12}. \end{equation}
Notice that $M_{11}$ is a Vandermonde matrix with determinant $\det M_{11}=\prod_{t=1}^s t^{s-t}$. One can also easily find
  \begin{equation}\label{mateq3} \det M_{12}=(-1)^{[s/2]} \quad\text{and}\quad \det M_2=\prod_{i=1}^sc_0(i).  \end{equation}
Now by \eqref{mateq0}, \eqref{mateq1}, \eqref{mateq2} and \eqref{mateq3},
 \[ \det M(s)=a^{s(s+1)} (-1)^{[s/2]}\prod_{i=1}^s \frac{i^{s-i}c_0(i)}{(ki-i-1)!(i+s)!}. \]
Hence $v(\det M(s))=0$.
\end{proof}

Denote $O(x)$ a number in $\overline{\Q}_p$ with valuation $\ge v(x)$ for $x\in\overline{\Q}_p$.
\begin{lemma}
$(i)$ For $i+j<d$, $F_{pi-j}=\g^{ki}(a_{ij}+O(\g))$.

$(ii)$ For $i+j\ge d$, $v(F_{pi-j})=ki-1$ and
  \[ F_{pi-(d-i)}=\frac{\g^{ki-1}(1+O(\g))}{(ki-1)!}.\]
\end{lemma}
\begin{proof}
Let
  \[  m=\begin{cases}
        ki-i-j,&\ \text{if}\ j<d-i;\\
        ki-i-j+d-1,    &\ \text{if}\ j\ge d-i,
  \end{cases}\]
  \[  n=\begin{cases}
        i+j,&\ \text{if}\ j<d-i;\\
        i+j-d,  &\ \text{if}\ j\ge d-i.
  \end{cases}\]
Then $pi-j=dm+(d-1)n$ and $0\le n\le d-1$. This lemma follows from
  \[F_{pi-j}=\sum_{l\ge 0} \g_{m-(d-1)l}\g_{n+dl}a^{n+dl}=\g_m\g_na^n(1+O(\g))=\frac{\g^{m+n}a^n}{m!n!}(1+O(\g)). \qedhere \]
\end{proof}

\begin{prop}
For $1\leq s\leq d-1$, the valuation of $\det A_1[s+1]$ is $w_0+w_1+\cdots+w_s$.
\end{prop}
\begin{proof}
Note that the first row of $A_1$ is $(1,0,0,\ldots)$. Let $A$ be the matrix by deleting the first row and column of $A_1[s+1]$. Then $\det A_1[s+1]=\det A$.

Let $D_1=\diag\{\g^{0/d},\g^{1/d},\ldots,\g^{s/d}\}$, $D_2=\diag\{\g^{k-1},\g^{2k-1},\ldots,\g^{(d-1)k-1}\}$ and $B[s]=(\g^{1-ki}F_{pi-j})_{1\le i,j\le s}$. Then $A=D_1^{-1}D_2B[s]D_1$. It suffices to compute $v(\det B[s])$.

Note that for $s=d-1$,
  \[ B[d-1]=\begin{pmatrix}
    \g a_{11}+O(\g^2)&\cdots&\g a_{1,d-2}+O(\g^2)                   &\frac{1+O(\g)}{(k-1)!}\\
    \vdots              &\udots&\frac{1+O(\g)}{(2k-1)!}  & b_{2,d-1} \\
    \g a_{d-2,1}+O(\g^2)&\udots&\udots&\vdots\\
    \frac{1+O(\g)}{((d-1)k-1)!}&b_{d-1,2}&\cdots&b_{d-1,d-1}
  \end{pmatrix} \]
with $v(b_{ij})=0$. If $1\le s\le \frac{d-1}{2}$, then
  \[B[s]=\begin{pmatrix}
      \g a_{11}+O(\g^2)&\cdots&\g a_{s1}+O(\g^2)\\
      \vdots&\ddots&\vdots\\
      \g a_{s1}+O(\g^2)&\cdots&\g a_{ss}+O(\g^2)\\
  \end{pmatrix}\]
has determinant
  \[ \det B[s]=\g^s(\det M(s)+O(\g)). \]
The valuation of $\det B[s]$ is $sv(\g)$.

If $\frac{d}{2}\le s\le d-1$, then
  \[B[s]=\begin{pmatrix}
      B[d-1-s] &P_1\\
      P_2&Q
  \end{pmatrix}.\]
The valuation of any entry of $B[d-1-s],P_1,P_2$ is $v(\g)$ and
  \[Q\equiv \begin{pmatrix}
    \multirow{2}{*}{{\Huge0}}&      &\frac{1}{(k-1)!}\\
                           &\udots&\\
    \frac{1}{((d-1)k-1)!}  &      &\multirow{2}*{{\Huge*}}
  \end{pmatrix}\bmod\g.\]
Thus $Q$ is invertible over the ring of integers of $\Q_p(\g)$. The determinant
  \[\det B[s]=\det Q\det(B[d-1-s]-P_1Q^{-1}P_2)=\det Q \det B[d-1-s](1+O(\g))\]
has valuation $(d-1-s)v(\g)$.

Finally, $A=D_1^{-1}D_2B[s]D_1$ has valuation
  \[(\sum_{i=1}^{s}(ki-1) +\min\{s,d-1-s\})v(\g)=w_0+w_1+\cdots+w_s. \qedhere \]
\end{proof}

\begin{proof}[Proof of Theorem~\ref{slope}]
For $1\leq s\leq d-1$, we have
  \[\begin{split}
    v(\det A_1[s+1])&=\sum_{i\le s} w_i\\
    &=\begin{cases}
      \frac{s(s+1)}{2d}+\frac{s(s+1)}{d(p-1)},&\ \text{if}\ s\le (d-1)/2;\\
      \frac{s(s+1)}{2d}+\frac{(d-s)(d-s-1)}{d(p-1)},&\ \text{if}\ s\ge d/2;
    \end{cases}\\
    &\le \frac{s(s+1)}{2d}+\frac{d^2-1}{4d(p-1)}.
    \end{split}\]
If $p> \frac{d^2+3}{4}$, then $\frac{d^2-1}{4d(p-1)}<1/d$. For $0\le t_0<t_1<\cdots<t_s$, assume $t_s\neq s$. Since
  \[v(F_{pi-j}\g^{(j-i)/d})\ge i/d,\]
we have
  \[ v(\det A_1[t_0,\ldots,t_s])\ge \frac{s^2+s+2}{2d}>v(\det A_1[s+1]). \]
Thus $v(c_{s+1})=v(\det A_1[s+1])=\sum_{i\le s}w_s$ and $\{w_0,w_1,\ldots,w_{d-1}\}$ are slopes of $\NP_p(\det(I-tA_1))$.

If moreover $p>\frac{d^2}{2}$, then $p\ge \frac{d^2+1}{2}$ and $\frac{d^2-1}{4d(p-1)}\le\frac{1}{2d}$. Choose $\beta_i=i/d$ in Theorem~\ref{zhu}, we have
 \[v(C_{s+1})=h(w_0+w_1+\cdots+w_s)\]
and
  \[\NP_q(\det(I-tA_h[d]))=\NP_p(\det(I-tA_1[d])).\]
Thus $w_0,w_1,\dots,w_{d-1}$ are $q$-adic slopes of $\NP_q(\detphi(I-tA_h))$.

By Theorem~\ref{Dwork},
  \[ \detphi(I-tA_h)=L^*(f,t)\detphi(I-tqA_h). \]
Since the valuation of any entry of $A_h$ is $\ge0$, the $q$-adic slopes of $\detphi(I-tA_h)$ are $\ge0$ and the $q$-adic slopes of $\detphi(I-tqA_h)$ are $\ge1$. Thus any $q$-adic slope of $\detphi(I-tA_h)$ less than $1$ must be a $q$-adic slope of $L^*(f,t)$. But $L^*(f,t)$ has degree $d$, hence $w_0,\ldots,w_{d-1}$ are all slopes of $L^*(f,t)$.
\end{proof}

\subsection{The case for general $\chi$}
Let $f(x)\in\F_q[x]$ be a polynomial with degree $d$. Assume $p\nmid d$. Let $\NP(f,x)$ be the piecewise linear function whose graph is the $q$-adic Newton polygon of $\det(I-tA_h)$. Let $\HP(f,x)$ be the piecewise linear function whose graph is the polygon with vertices
  \[ (k,\frac{k(k-1)}{2d}),\quad k=0,1,2,\ldots. \]
Then $\NP(f,x)\ge \HP(f,x)$ (cf. \cite{LW,OY}). Set
  \[ \gap(f)=\max_{x\ge 0}\{\NP(f,x)-\HP(f,x)\}. \]

\begin{thm}(See \cite[Theorem~4.3]{OY}.)
Let $0=\alpha_0<\alpha_1<\cdots<\alpha_{d-1}<1$ denote the slopes of the $q$-adic Newton polygon of $L^*(f,t)$. If $\gap(f)<1/h$, then the $q$-adic Newton polygon of $L^*(f,\chi,t)$ has slopes
  \[\{p^{1-m_\chi}(i+\alpha_j):0\le i\le p^{m_\chi-1}-1,0\le j\le d-1\}\]
for any non-trivial finite character $\chi$.
\end{thm}
\begin{proof}[Proof of Theorem~\ref{chislope}]
The slopes of $\NP(f,x)$ are
  \[\{i+w_j:i\ge0,0\le j\le d-1\}.\]
Notice that
  \[\sum_{i=0}^{d-1} w_i=\sum_{i=0}^{d-1}\frac{i}{d},\]
$\NP(f,x)-\HP(f,x)$ is a periodic function with period $d$. For $0\le k<d$,
  \[\NP(f,x)-\HP(f,x)\le\sum_{i\le(d-1)/2}\frac{2i}{d(p-1)}\le\frac{d^2-1}{4d(p-1)}.\]
If $p>\frac{h(d^2-1)}{4d}+1$, then $\gap(f)<1/h$ and this concludes the proof.
\end{proof}

\textbf{Acknowledgements.} This paper was prepared when the authors were visiting the Academy of Mathematics and Systems Science and the Morningside Center of Mathematics of Chinese Academy of Sciences. We would like to thank Professor Ye Tian for his hospitality. We also would like to thank Jinbang Yang for many helpful discussions. This work was partially supported by NKBRPC(2013CB834202) and NSFC(11171317 and 11571328).

\end{document}